\documentclass[a4paper]{article}
\usepackage{amsfonts}
\usepackage{amsmath}
\usepackage{indentfirst,latexsym,bm}
\usepackage{amssymb}
\usepackage{amsmath}
\usepackage{graphics}
\usepackage{fancyvrb}
\usepackage{color}
\usepackage{mathrsfs}
\usepackage{hyperref}

\numberwithin{equation}{section}

\newtheorem{theorem}{Theorem}

\newtheorem{assumption}{Assumption}

\newtheorem{condition}{Condition}

\newtheorem{definition}[theorem]{Definition}

\newtheorem{lemma}[theorem]{Lemma}

\newtheorem{remark}[theorem]{Remark}

\newenvironment{proof}[1][Proof]{\textbf{#1.} }{\ \rule{0.5em}{0.5em}}
\begin{document}

\title{{A Functional Approach to FBSDEs and Its Application in Optimal Portfolios}}

\date{}
\author{\emph{{\small {\textsc{By G. Liang\ \ T. Lyons\ \ and\ \ Z. Qian}}}} \\
\emph{{\small {University of Oxford, U.K.}}}} \maketitle
\begin{abstract}

In Liang et al (2009), the current authors demonstrated that BSDEs
can be reformulated as functional differential equations, and as an
application, they solved BSDEs on general filtered probability
spaces. In this paper the authors continue the study of functional
differential equations and demonstrate how such approach can be used
to solve FBSDEs. By this approach the equations can be solved in one
direction altogether rather than in a forward and backward way. The
solutions of FBSDEs are then employed to construct the weak
solutions to a class of BSDE systems (not necessarily scalar) with
quadratic growth, by a nonlinear version of Girsanov's
transformation. As the solving procedure is constructive, the
authors not only obtain the existence and uniqueness theorem, but
also really work out the solutions to such class of BSDE systems
with quadratic growth. Finally an optimal portfolio problem in
incomplete markets is solved based on the functional differential
equation
approach and the nonlinear Girsanov's transformation.\\




\end{abstract}
\leftskip0truecm \rightskip0truecm
\newpage
\section{Introduction}

Backward stochastic differential equations (BSDE) provide a new
perspective to look at the infinitesimal behavior of Markov
processes, and they have been found intrinsically linked to a class
of nonlinear partial differential equations (PDEs). In a fundamental
paper by Pardoux sand Peng \cite{MR1037747}, they solved a class of
nonlinear BSDEs with Lipschitz drivers for the first time. An
intrinsic connection to nonlinear PDEs, which is now well known as
nonlinear Feynman-Kac formula, was later established by Peng
\cite{MR1149116} and Pardoux and Peng \cite{MR1176785}. Its
applications in finance were discovered by Duffie and Epstein
\cite{Duffie} and El Karoui et al \cite{MR1434407}. For a more
comprehensive review of the BSDE theory, we refer to \cite{El
Karoui}, \cite{MR1752671} and \cite{MR1696772} and reference
therein.

In Liang et al \cite{Liang}, the current authors demonstrate that
BSDEs can be reformulated as functional differential equations. As
an application we can solve BSDEs on general filtered probability
spaces, and in particular without the requirement of martingale
representation. In this paper we try to apply such \emph{functional
differential equation approach}, or \emph{functional approach} for
short, to forward backward stochastic differential equations
(FBSDEs), and demonstrate how it can be used to solve a financial
optimal portfolio problem in incomplete markets.

Let us recall such idea, which is the first main ingredient of the
paper. Let $(\Omega,\mathcal{F},\mathcal{F}_t,\mathbf{P})$ be a
filtered probability space satisfying the usual conditions. Given a
semimartingale $(Y_t)_{t\in[0,T]}$ with the following decomposition:
$$Y_t=M_t-V_t,\ \ \ \text{for}\ t\in[0,T],$$
where $M$ is local martingale and $V$ is a finite variation process,
if we further know the terminal data of $Y$, say $Y_T=\xi$ for some
$\mathcal{F}_T$-measurable random variable $\xi$, then we also have
$\xi=M_T-V_T$. Therefore
\begin{equation}\label{chap1:CrucialRelation}
\left\{
\begin{array}{ll}
Y_t=E[\xi+V_T|\mathcal{F}_t]-V_t,&\\[+0.1cm]
M_t=E[\xi+V_T|\mathcal{F}_t],&\text{for}\ t\in[0,T].
\end{array}
\right.
\end{equation}

Before we apply the above relationship (\ref{chap1:CrucialRelation})
to a specific FBSDE setting, let us first look at it from potential
theory point of view. For a given domain $D\subset R^d$, a
real-valued function $f$ is called superharmonic in $D$ if
$$\int_{\partial B(0,r)}f(x+y)\sigma_r(dy)\leq f(x),\ \ \
\text{for}\ x\in D\ \text{and}\ r<dist(x,\partial D),$$ where
$\sigma_r$ is the measure on the surface of the ball $\partial
B(0,r)$, normalized to have the total mass $1$. $f$ is called
harmonic if furthermore the equality holds. If $f$ is superharmonic
in $D$, then there exists a unique positive Borel measure on $D$
such that the following Riesz decomposition holds:
$$f(x)=h(x)+G_{D}\mu(x),\ \ \ \text{for}\ x\in D,$$
where $h$ is harmonic in $D$ and $G_D\mu$ is the Green's potential
of $\mu$ on $D$, i.e. $G_D\mu(x)=\int_{D}G(x,y)\mu(dy)$ with
$G(x,y)$ as the Green's function of the Laplace equation in $R^d$.
Given the boundary data of such superharmonic function $f$, by the
above Riesz decomposition, the Green's potential $G_D\mu$ is often
used to study $f$.

The probabilistic counterpart of the above Riesz decomposition is
the Doob-Meyer decomposition. For a supermartingale
$(Y_t)_{t\in[0,T]}$ with C\`adl\`ag sample paths, the Doob-Meyer
decomposition says that there exists a unique increasing predictably
measurable process $V$ starting from $V_0=0$ such that $M$ defined
by $M_t=Y_t+V_t$ for $t\in[0,T]$ is a martingale. The above
relationship (\ref{chap1:CrucialRelation}) we just established tells
us exactly the same thing as in the potential theory: given the
terminal data of a supermartingale (which is a semimartingale)
defined on $[0,T]$, we can study such supermartingale by
investigating the increasing predictably measurable process $V$.

Now we apply the above relationship (\ref{chap1:CrucialRelation}) to
a specific FBSDE, which will be served as an auxiliary equation
later. For any given filtered probability space
$(\Omega,\mathcal{F},\mathcal{F}_t,\mathbf{Q})$ satisfying the usual
conditions, with a $d$-dimensional Brownian motion
$B=(B^1,\cdots,B^d)^{*}$, (The superscript $^*$ denotes the matrix
transposition)
\begin{eqnarray}\label{MainFBSDE1}
    \left\{
    \begin{array}{ll}
    dX_t=f(t,Y_t,Z_t)dt+dB_t,\\[+0.2cm]
    X_0=x\in R^d,\\[+0.2cm]
    dY_t=-h(t,Y_t,Z_t)dt+Z_tdB_t,\\[+0.2cm]
    Y_T=\phi(X_T).
    \end{array}\right.
\end{eqnarray}
with the coefficients $f:[0,T]\times R^n\times R^{n\times
d}\rightarrow R^d$ , $h:[0,T]\times R^n\times R^{n\times
d}\rightarrow R^n$ and $\phi: R^d\rightarrow R^n$ satisfying
Condition \ref{Cond2.4} to be introduced later.

To solve FBSDE (\ref{MainFBSDE1}), for $\tau\in[0,T]$, we consider
the following functional differential equation on $[\tau,T]$:
\begin{equation}\label{MainFDE1}
V_t=\int_{\tau}^th(s,Y(V,X)_s,Z(V,X)_s)ds
\end{equation}
together with the forward process $X$:
\begin{equation}\label{MainSDE1}
X_t=x+\int_{\tau}^tf(s,Y(V,X)_s,Z(V,X)_s)ds+\int_{\tau}^tdB_s,
\end{equation}
where
\begin{equation}\label{relationship}
\left\{
\begin{array}{ll}
Y(V,X)_t=E[\phi(X_T)+V_T|\mathcal{F}_t]-V_t,\\[+0.2cm]
\displaystyle\int_{\tau}^TZ(V,X)_sdB_s=\phi(X_T)+V_T-E[\phi(X_T)+V_T|\mathcal{F}_{\tau}].
\end{array}
\right.
\end{equation}

If we can solve $(V,X)$ for functional differential equation
(\ref{MainFDE1}) (\ref{MainSDE1}) with $\tau=0$, then by Lemma
\ref{LemmaEQU} in Section 2.2, $(Y(V,X),Z(V,X),X)$ will provide us
with the solution to (\ref{MainFBSDE1}). For the notation's
simplicity, let us denote $\Psi=(V,X)^{*}$ and $F=(h,f)^{*}$. Then
(\ref{MainFDE1}) (\ref{MainSDE1}) are simplified to:
\begin{equation}\label{MainFDE2}
\Psi_t=\chi_x+\int_{\tau}^tF(s,Y(\Psi)_s,Z(\Psi)_s)ds+\int_{\tau}^t\chi_1dB_s
\end{equation}
where $\chi_x\in R^{n+d}$ with the first $n$ components being $0$
and next $d$ components being $x\in R^{d}$. Note that by functional
differential equation (\ref{MainFDE2}), we can solve all the
components $(Y,Z,X)$ in one direction altogether. We consider the
solutions to (\ref{MainFDE2}) in the following space:
\begin{itemize}
  \item $\mathcal{C}([0,T];R^{n+d})$: the space of continuous and
$\mathcal{F}_t$-adapted processes $(\Psi_t)_{t\in[0,T]}$ valued in
$R^{n+d}$ such that
$\sup_{t\in[0,T]}|\Psi_t|\in\mathcal{L}^2(\Omega,\mathcal{F}_T,\mathbf{Q})$
and endowed with the following norm:
$$||\Psi||_{\mathcal{C}[0,T]}=\sqrt{E\sup_{t\in[0,T]}|\Psi_t|^2}.$$
\end{itemize}

In Section 2.2 we will mainly solve functional differential equation
(\ref{MainFDE2}) and prove the following theorem:

\begin{theorem} If the coefficients satisfy Condition \ref{Cond2.4},
then there exists a unique solution $\Psi\in
\mathcal{C}([0,T];R^{n+d})$ to functional differential equation
(\ref{MainFDE2}).
\end{theorem}

The second main ingredient of this paper is a nonlinear version of
Girsanov's transformation, which is employed to connect the above
FBSDE (\ref{MainFBSDE1}) and a class of BSDE systems with quadratic
growth. Namely we will use strong solutions of FBSDEs to construct
weak solutions to a class of BSDE systems with quadratic growth.

In the PDE theory, if the nonlinear terms in equations have at most
quadratic growth with respect to the gradient of solutions, the
nature of the equations completely change. In the BSDE theory, there
is a class of BSDEs with quadratic growth corresponding to such
PDEs, and they are usually called quadratic BSDEs. The study of
quadratic BSDEs was initialized by Kobylanski \cite{MR1782267} using
the idea of the Cole-Hopf transformation adapted from the PDE
theory. Her result was substantially developed and generalized by
Briand and Hu \cite{MR2257138} and \cite{MR2391164}, where they
extended to the equations with the unbounded terminal data and with
the convex driver. On the other hand, Quadratic BSDEs have found a
lot of applications in finance. For example, they appear naturally
when one wants to derive the value function for the maximization of
expected utility, use indifference pricing idea to hedge a
contingent claim written on nontradeable underlying assets, or
consider the risk measure. See for example \cite{Jeanblanc}
\cite{Martin} \cite{MR2152241} \cite{Mania} \cite{MR2465489} and
\cite{MR1802922}.

In this paper we mainly consider the following quadratic BSDE system
(not necessarily scalar):
\begin{eqnarray}\label{MainqBSDE1}
    \left\{
    \begin{array}{ll}
    dY_t=-h(t,Y_t,Z_t)dt-Z_tf(t,Y_t,Z_t)dt+Z_tdW_t,\\[+0.2cm]
    Y_T=\phi(W_T)
    \end{array}\right.
\end{eqnarray}
where $W=(W^1,\cdots,W^d)^{*}$ is a $d$-dimensional Brownian mtion
starting from $x\in R^d$. The coefficients $h$, $f$ and $\phi$ are
supposed to satisfy the following condition:

\begin{condition}\label{Cond2.4}
All the coefficients $h:[0,T]\times R^n\times R^{n\times
d}\rightarrow R^n$, $f:[0,T]\times R^n\times R^{n\times
d}\rightarrow R^d$ and $\phi: R^d\rightarrow R^n$ are continuous.
Moreover $h$, $f$ and $\phi$ are Lipschitz continuous, i.e.
\begin{eqnarray*}
    \left\{
    \begin{array}{ll}
|h(t,y,z)-h(t,\bar{y},\bar{z})|\leq
C_1(|y-\bar{y}|+|z-\bar{z}|),\\[+0.2cm]
|f(t,y,z)-f(t,\bar{y},\bar{z})|\leq
C_1(|y-\bar{y}|+|z-\bar{z}|),\\[+0.2cm]
|\phi(x)-\phi(\bar{x})|\leq C_2|x-\bar{x}|,
\end{array}\right.
\end{eqnarray*}
and $\phi$ is uniformly bounded,
$$\sup_{x\in R^d}|\phi(x)|\leq M,$$
for $t\in[0,T]$, $y,\bar{y}\in R^n$, $z,\bar{z}\in R^{n\times d}$
and $x,\bar{x}\in R^d$.
\end{condition}

Because of the terms with the coefficient $f$, the equations have at
most quadratic growth, i.e. there exists a constant $C_3$ such that
for any $y\in R^n$ and $z\in R^{n\times d}$,
$$|zf(t,y,z)|\leq C_3|z|(t+|y|+|z|).$$
The quadratic growth term in (\ref{MainqBSDE1}) is more special than
the usual one considered in the literature. However this special
structure is enough to cover the most examples of quadratic BSDEs
known in finance, at least with some extra conditions added. We will
consider one specific example from optimal portfolio problems in
Section 4. Moreover, all the excising results of quadratic BSDEs are
only for the case $n=1$. The current paper seems to be the first
attempt to consider the quadratic BSDE systems.

On the other hand, one may wonder why the terminal data has the
special form $\phi(W_T)$. This is only for the presentation's
simplicity. The whole paper's results can be extended without
difficulty to the case $\phi(X_T)$ where $X$ is driven by stochastic
differential equations (SDEs):
$$dX_t^i=X_t^i(b_t^idt+\sigma_t^idW_t),\ \ \ \text{for}\ i=1,\cdots,m.$$
with the coefficients $b^i$ and $\sigma^i$ satisfying certain
regularity conditions.

To solve (\ref{MainqBSDE1}), we will pursue another direction
different from the existing method for quadratic BSDEs. Namely we
don't use the Cole-Hopf transformation at all and don't assume the
underlying probability space and Brownian motion as any given;
instead we consider weak solutions of quadratic BSDEs.

Before presenting the definition of weak solutions to
(\ref{MainqBSDE1}), let us mention some already existing work about
weak solutions. One of the first attempts to introduce the weak
solutions for BSDEs was Buckdahn et al \cite{MR2141331}, and
Buckdahn and Engelbert \cite{MR2354579} further proved the
uniqueness of their weak solutions. However the driver of their BSDE
does not evolve the martingale representation part $Z$. On the other
hand, the notion of weak solutions for FBSDEs was introduced by
Antonelli and Ma \cite{MR1978231} and further developed by Delarue
and Guatteri \cite{MR2307056}, and by Ma et al \cite{MR2478677} and
Ma and Zhang \cite{MaJin1} who employed the martingale problem
approach.

\begin{definition}\label{Def}
A weak solution to BSDE (\ref{MainqBSDE1}) is a triple
$(\Omega,\mathcal{F},\mathbf{P}^x)$, $\{\mathcal{F}_t\}$ and
$(Y,Z^{\mathbf{P}^x},W)$ such that

(1) $(\Omega,\mathcal{F},\mathbf{P}^x)$ is a complete probability
space with the filtration $\{\mathcal{F}_t\}$ satisfying the usual
conditions;

(2) under such filtered probability space, $Y,Z^{\mathbf{P}^x}$ and
$W$ are $\mathcal{F}_t$-adapted, and $Y$ is a special
semimartingale, $Z^{\mathbf{P}^x}$ is the density representation of
$Y$ under $\mathbf{P}^x$, and $W$ is a Brownian motion starting from
$\mathbf{P}^x(W_0=x)=1$;

(3) The increments $\{W_u-W_t: t\leq u\leq T\}$ must be independent
of the $\sigma$-algebra $\mathcal{F}_t$;

(4) the following integral equation satisfies:
\begin{equation}\label{MainqBSDE2}
Y_t=\phi(W_T)+\int_t^Th(s,Y_s,Z^{\mathbf{P}^x}_s)ds+\int_t^T
Z^{\mathbf{P}^x}_sf(s,Y_s,Z^{\mathbf{P}^x}_s)ds-\int_t^T
Z^{\mathbf{P}^x}_sdW_s.
\end{equation}
\end{definition}

\begin{remark} In this paper, the density representation $Z^{\mathbf{P}^x}$
means it is the density representation for the martingale part of
the special semimartingale $Y$, and we use the superscript
$\mathbf{P}^x$ to emphasize the dependency of the density
representation on the probability measure $\mathbf{P}^x$.
\end{remark}

\begin{remark} Our definition of weak solutions is more related to
Buckdahn et al \cite{MR2141331}. The filtration $\{\mathcal{F}_t\}$
plays an important role here. If $\mathcal{F}_t=\mathcal{F}_t^W$,
i.e. the filtration is generated by the Brownian motion $W$
augmented by the $\mathbf{P}^x$-null sets in $\mathcal{F}$, the
solution turns to be a strong solution. In Ma and Zhang
\cite{MaJin1}, such solution is also called a semi-strong solution.
Actually the smallest filtration for weak solutions is the
filtration $\{\mathcal{F}^{W,Y,Z}_t\}$ generated by $W,Y,Z$ and
satisfying the usual conditions.
\end{remark}

\begin{remark}
Condition (3) automatically holds given the Brownian motion $W$ with
the filtration $\{\mathcal{F}_t\}$. In fact such condition simply
means $\{\mathcal{F}_t\}$ consists, additionally to
$\{\mathcal{F}_t^{W}\}$, only of independent experiments. In
Buckdahn et al \cite{MR2141331}, such condition is formulated in
terms of martingales, i.e. any $\mathcal{F}_t^W$-martingale must be
an $\mathcal{F}_t$-martingale. In Kurtz \cite{MR2336594}, such kind
of condition is called the compatibility constraint. (3) is
extremely useful when we want to identify weak solutions are strong
solutions.
\end{remark}

For the notation's simplicity, we will suppress the superscript $x$
of $\mathbf{P}^x$ from now on if no confusion may arise. Now we
describe our idea formally before presenting the existence and
uniqueness theorem of BSDE (\ref{MainqBSDE1}). The basic idea is
using the strong solution of FBSDE to construct the weak solution to
quadratic BSDE. Let us start with a Brownian motion family $B$ on
$(\Omega,\mathcal{F},\mathbf{Q})$ with the filtration
$\{\mathcal{F}_t\}$ satisfying the usual conditions, and consider
FBSDE (\ref{MainFBSDE1}).

Suppose FBSDE (\ref{MainFBSDE1}) admits a unique solution
$(X,Y,Z^{\mathbf{Q}})$. Then we define a new probability measure
$\mathbf{P}$ by
$$\frac{d\mathbf{P}}{d\mathbf{Q}}=\mathscr{E}(N)$$
where $\mathscr{E}(N)$ is the Dol\'eans-Dade exponential of $N$ with
$$N=-\int_0^{\cdot}\langle f(s,Y_s,Z_s^{\mathbf{Q}}), dB_s\rangle_d$$
where $\langle\cdot,\cdot\rangle_d$ denotes the inner product in
$R^d$. Under the new probability measure $\mathbf{P}$, by the
Girsanov's theorem, $B$ has the following decomposition:
\begin{align*}
B=&\ (B-[B,N])+[B,N],\\
=&\
\left(B+\int_0^{\cdot}f(s,Y_s,Z_s^{\mathbf{Q}})ds\right)-\int_0^{\cdot}f(s,Y_s,Z_s^{\mathbf{Q}})ds
\end{align*} where
$B-[B,N]=B+\int_0^{\cdot}f(s,Y_s,Z_s^{\mathbf{Q}})ds$ is a
martingale under $\mathbf{P}$, and furthermore by the L\'evy's
characterization, it is in fact a Brownian motion under
$\mathbf{P}$. We further define $W$ by $W=x+B-[B,N]$.

Under the probability measure $\mathbf{P}$ and with the Brownian
motion $W$, let's rewrite the backward equation in FBSDE
(\ref{MainFBSDE1}):
$$dY_t=-h(t,Y_t,Z_t^{\mathbf{Q}})dt-
Z_t^{\mathbf{Q}}f(t,Y_t,Z_t^{\mathbf{Q}})dt+ Z_t^{\mathbf{Q}}dW_t$$
with $Y_T=\phi(W_T)$. If we can prove
$Z^{\mathbf{Q}}=Z^{\mathbf{P}}$, then triple
$(\Omega,\mathcal{F},\mathbf{P})$, $\{\mathcal{F}_t\}$ and
$(Y,Z^{\mathbf{P}},W)$ is just one weak solution we want to find.

There are mainly three steps needed to be verified for the above
solving procedure. The first step is about the invariant property of
the density representation under the change of probability measure,
i.e. $Z^{\mathbf{Q}}=Z^{\mathbf{P}}$; The second step is of course
the solvability of FBSDE (\ref{MainFBSDE1}); the last step is the
Dol\'eans-Dade exponential $\mathscr{E}(N)$ must be a
uniform-integrable martingale to guarantee $\mathbf{P}$ is a
probability measure. As long as the above three steps are verified,
we have the following theorem:

\begin{theorem}\label{maintheorem}
If the coefficients satisfy Condition \ref{Cond2.4}, then there
exists at least one weak solution $(\Omega,\mathcal{F},\mathbf{P})$,
$\{\mathcal{F}_t\}$ and $(Y,Z^{\mathbf{P}},W)$ to BSDE
(\ref{MainqBSDE1}).
\end{theorem}

The paper is organized as follows: In section 2 we verify the above
three steps and prove Theorem \ref{maintheorem}, while in Section 3
the uniqueness and the connection between weak solutions and strong
solutions are discussed. Finally we apply our method to an optimal
portfolio problem in incomplete markets in Section 4.

\section{Weak solutions and existence}
\subsection{Invariant property of density representation}

The following lemma is almost trivial but crucial to our results,
which states that the density representation of a special
semimartingale is invariant under the equivalent change of
probability measure. This observation is firstly made in Liang et al
\cite{Qian1}, and we recall it here for completeness.

\begin{lemma}\label{Lemma1} Let $B$ be a Brownian motion on $(\Omega, \mathcal{F},
\mathbf{Q})$ with the filtration $\{\mathcal{F}_t\}$ satisfying the
usual conditions. Let $Z^{\mathbf{Q}}$ be the density representation
of a special semimartingale $Y$ under $\mathbf{Q}$. If define an
equivalent probability measure $\mathbf{P}$ by
$\frac{d\mathbf{P}}{d\mathbf{Q}}=\mathscr{E}(N)$ for some uniform
integrable martingale $\mathscr{E}(N)$, then
$Z^{\mathbf{P}}_t=Z^{\mathbf{Q}}_t$ for $a.e.$ $t\in[0,T]$, $a.s.$
\end{lemma}

\begin{proof} Under the probability measure $\mathbf{Q}$,
$Y$ has the canonical decomposition $Y=M-V$ with $M$ being a local
martingale and $V$ being a finite variation process, and moreover,
$M$ admits the martingale representation:
\begin{equation}\label{Rep1}
M_t-M_0=\int_0^tZ^{\mathbf{Q}}_sdB_s,\ \ \ \text{for}\ t\in[0,T],\
a.s.
\end{equation}
for some predictably measurable process $Z^{\mathbf{Q}}$.

By Girsanov's theorem, $Y$ is still a special semimartingale under
$\mathbf{P}$ but with the canonical decomposition
$Y=\bar{M}-\bar{V}$, where $\bar{M}=M-[M,N]$ is a martingale, and
$\bar{V}=\bar{M}-Y$ is a finite variation process. We also have
$\bar{B}=B-[B,N]$ as a Brownian motion under $\mathbf{P}$. Hence
under $\mathbf{P}$, (\ref{Rep1}) becomes
$$\bar{M}_t-\bar{M}_0+[M,N]_t=\int_0^tZ^\mathbf{Q}_sd\bar{B}_s+\int_0^tZ^\mathbf{Q}_sd[B,N]_s,\
\ \ \text{for}\ t\in[0,T],\ a.s..$$ By identifying the martingale
parts and finite variation parts of the above equality, we must have
$$\bar{M}_t-\bar{M}_0=\int_0^tZ^\mathbf{Q}_sd\bar{B}_s,\ \ \ \text{for}\ t\in[0,T],\
a.s..$$ But on the other hand under $\mathbf{P}$ we also have
$$\bar{M}_t-\bar{M}_0=\int_0^tZ^\mathbf{P}_sd\bar{B}_s,\ \ \ \text{for}\ t\in[0,T],\
a.s.$$ for some predictably measurable process $Z^{\mathbf{P}}$, so
$$\int_0^T|Z^\mathbf{P}_s-Z^\mathbf{Q}_s|^2ds=0,\ \ \
a.s.,$$ which proves the claim.
\end{proof}

Since the density representation usually determines the hedging (or
replicating) strategy in finance, a direct consequence of Lemma
\ref{Lemma1} is that hedging strategy is independent of the choices
of the equivalent (martingale) probability measures. Due to Lemma
\ref{Lemma1}, we will not emphasize the dependency of the density
representation on the probability measure, and simply write it as
$Z$ from now on.

\subsection{Functional approach to FBSDEs}

The study of FBSDEs was initiated by Antonelli \cite{MR1233625}, and
this subject was further developed in \cite{MR1355060}
\cite{MR1262970} \cite{MR1701517} \cite{MR1675098} \cite{MR1440146}
and especially the monograph \cite{MR1704232} by Ma and Yong.
However most of them either solved the equations locally or assumed
some regularity on the coefficients (e.g. smoothness and
monotonicity). Recently Delarue \cite{Delarue} solved FBSDEs
globally with Lipschitz continuous assumptions on the coefficients
by combining the method of contraction mapping and the four-step
scheme of FBSDEs.

In this subsection we try to reformulate FBSDE (\ref{MainFBSDE1}) as
functional differential equation (\ref{MainFDE2}) and solve such
functional differential equation instead. The approach may benefit
especially numerical solutions of FBSDEs, because a usual obstacle
to numerically solve FBSDEs is one needs to solve $(Y,Z)$ backwards
and $X$ forwards at the same time. By introducing a functional
differential equation, we can solve all the components $(Y,Z,X)$ in
one direction altogether.

We first establish the equivalence between FBSDE (\ref{MainFBSDE1})
and functional differential equation (\ref{MainFDE2}). Besides the
space $\mathcal{C}([0,T];R^n)$ we further introduce the following
space:

\begin{itemize}
\item $H^2([0,T];R^{n})$: the space of predictably
measurable processes endowed with the norm:
$$||Z||_{H^2[0,T]}=\sqrt{E\int_0^T|Z_s|^2ds}.$$
\end{itemize}

\begin{lemma}\label{LemmaEQU}
FBSDE (\ref{MainFBSDE1}) admits a unique solution $(Y,Z,X)\in
\mathcal{C}([0,T];R^n)\times H^2([0,T];R^{n\times
d})\times\mathcal{C}([0,T];R^d)$ if and only if functional
differential equation (\ref{MainFDE2}) admits a unique solution
$\Psi\in\mathcal{C}([0,T];R^{n+d})$, and therefore by Theorem 1,
FBSDE (\ref{MainFBSDE1}) admits a unique solution.
\end{lemma}

\begin{proof} Suppose $(Y,Z,X)$ is the unique solution to FBSDE
(\ref{MainFBSDE1}). Since $Y$ is a special (continuous)
semimartingale, it admits the following canonical decomposition:
$Y_t=M_t-V_t$ for $t\in[0,T]$, where $M$ is a continuous local
martingale and $V$ is a continuous finite variation process with
$V_0=0$. Furthermore by the martingale representation
$M_t=M_0+\int_0^tZ_sdB_s$ for $t\in[0,T]$, we have
\begin{equation}\label{canonical}
Y_t=M_0+\int_0^tZ_sdB_s-V_t,\ \ \ \text{for}\ t\in[0,T],
\end{equation}
from which we obtain the relationship (\ref{relationship}). The
backward equation in FBSDE (\ref{MainFBSDE1}) becomes
$$M_0+\int_0^tZ_sdB_s-V_t=\phi(X_T)+\int_t^Tf(s,Y_s,Z_s)ds-\int_t^TZ_sdB_s.$$
By taking conditional expectation with respect to $\mathcal{F}_t$ on
both sides:
\begin{align*}
&M_0+\int_0^tZ_sdB_s-V_t\\
=&\ E[\phi(X_T)|\mathcal{F}_t]+E\left[\int_t^Tf(s,Y_s,Z_s)ds|\mathcal{F}_t\right]\\
=&\
E[\phi(X_T)|\mathcal{F}_t]+E\left[\int_0^Tf(s,Y_s,Z_s)ds|\mathcal{F}_t\right]-\int_0^tf(s,Y_s,Z_s)ds.
\end{align*}
By the uniqueness of the canonical decomposition of $Y$, and by
identifying the martingale part and finite variation part of both
sides, $V$ must satisfy (\ref{MainFDE1}).

On the other hand, if $\Psi$ is the unique solution to functional
differential equation (\ref{MainFDE2}), the relationship
(\ref{relationship}) can be rewritten as
\begin{align*}
Y(V,X)_t&=E[\phi(X_T)+V_T|\mathcal{F}_t]-V_t\\
&=E[\phi(X_T)+V_T|\mathcal{F}_t]-\int_0^tf(s,Y(V,X)_s,Z(V,X)_s)ds,
\end{align*}
together with
$$\int_t^{T}Z_sdB_s=\phi(X_T)+V_T-E[\phi(X_T)+V_T|\mathcal{F}_t],$$
from which we deduce $(Y(V,X),Z(V,X),X)$ must satisfy
(\ref{MainFBSDE1}).\end{proof}\\

The rest of this subsection is devoted to the proof of Theorem 1.

\begin{lemma}\label{lemmaforlocal} If the coefficients satisfy
Condition \ref{Cond2.4}, and $\tau$ satisfies
$$\sqrt{T-\tau}\leq\frac{1}{8C_1(1+C_2)}\wedge1,$$
then functional differential equation (\ref{MainFDE2}) admits a
unique solution $\Psi\in\mathcal{C}([\tau,T];R^{n+d})$.
\end{lemma}

\begin{proof} The mapping defined by (\ref{MainFDE2}) is denoted by
$\mathbb{L}$. We will first show that
$\mathbb{L}:\mathcal{C}([\tau,T];R^{n+d})\rightarrow\mathcal{C}([\tau,T];R^{n+d})$.
In fact for $\Psi\in\mathcal{C}([\tau,T];R^{n+d})$,
\begin{align*}
&||\mathbb{L}(\Psi)||_{\mathcal{C}[\tau,T]}\\
\leq&\ |x|+
\sqrt{E\left(\int_{\tau}^T|F(s,Y(\Psi)_s,Z(\Psi)_s)|ds\right)^2}+
\sqrt{E\sup_{\tau\leq t\leq T}|\int_{\tau}^t\chi_1dB_s|^2}\\
\leq&\ |x|+
\sqrt{T-\tau}\sqrt{E\left(\int_{\tau}^T|F(s,Y(\Psi)_s,Z(\Psi)_s)|^2ds\right)}
+2\sqrt{E|\int_{\tau}^T\chi_1dB_s|^2}\\
\leq&\ |x|+(C_1\sqrt{T-\tau}+2d)\sqrt{\int_{\tau}^T(s^2\vee 1)ds}\\
&+C_1\sqrt{T-\tau}\sqrt{E\int_{\tau}^T|Y(\Psi)_s|^2ds}+
C_1\sqrt{T-\tau}\sqrt{E\int_{\tau}^T|Z(\Psi)_s|^2ds}.
\end{align*}
Note that
\begin{align*}
&\sqrt{E\int_{\tau}^T|Y(\Psi)_s|^2ds}\\
\leq&\
\sqrt{E\int_{\tau}^T\{E[\phi(X_T)|\mathcal{F}_s]\}^2ds}+\sqrt{E\int_{\tau}^T\{E[V_T|\mathcal{F}_s]\}^2ds}+\sqrt{E\int_{\tau}^T|V_s|^2ds}\\
\leq&\
\sqrt{\int_{\tau}^TE|\phi(X_T)|^2ds}+\sqrt{\int_{\tau}^TE|V_T|^2ds}+\sqrt{\int_{\tau}^TE|V_s|^2ds}\\
\leq&\ \sqrt{T-\tau}(2+C_2)||\Psi||_{C[\tau,T]},
\end{align*}
and by It\^o's isometry,
\begin{align*}
\sqrt{E\int_{\tau}^T|Z(\Psi)_s|^2ds}=&\
\sqrt{E\left(\int_{\tau}^TZ(\Psi)_sdB_s\right)^2}\\
\leq&\
\sqrt{E|\phi(X_T)|^2}+\sqrt{E|V_T|^2}+\sqrt{E\left\{E[\phi(X_T)|\mathcal{F}_{\tau}]\right\}^2}\\
&+\sqrt{E\left\{E[V_T|\mathcal{F}_{\tau}]\right\}^2}\\
\leq&\ (2+2C_2)||\Psi||_{C[\tau,T]}.
\end{align*}
Therefore $||\mathbb{L}(\Psi)||_{\mathcal{C}[\tau,T]}<\infty$.
Similarly for $\Psi,\Psi^{'}\in\mathcal{C}([\tau,T];R^{n+d})$, we
have
\begin{align*}
&||\mathcal{L}(\Psi)-\mathcal{L}(\Psi^{'})||_{\mathcal{C}[\tau,T]}\\
\leq&\
C_1\sqrt{T-\tau}\left(\sqrt{T-\tau}(2+C_2)+2+2C_2\right)||\Psi-\Psi^{'}||_{\mathcal{C}[\tau,T]}\\
\leq &\ \frac12||\Psi-\Psi^{'}||_{\mathcal{C}[\tau,T]}
\end{align*}
by the condition on $\tau$. Hence $\mathbb{L}$ defined by
(\ref{MainFDE2}) is a contraction mapping on
$\mathcal{C}([\tau,T];R^{n+d})$.
\end{proof}\\

Based on Lemma \ref{lemmaforlocal}, we next extend to the global
solution on $[0,T]$. To do this we pursue the bounded solutions for
FBSDE (\ref{MainFBSDE1}). First by the Markov property, there exists
a Borel-measurable function $\Phi$ such that $Y_t=\Phi(t,X_t)$. By
checking the proof for Lemma \ref{lemmaforlocal}, the crucial step
to extend to the global solution of (\ref{MainFDE2}) is that one
needs a uniform estimate for the gradient of $\Phi$. We recall a
regularity result from Delarue \cite{MR2053051}.

\begin{lemma}\label{Delarue} (Delarue \cite{MR2053051})
Under Condition \ref{Cond2.4} on the coefficients, there exists a
Borel measurable $\Phi$ such that $Y_t=\Phi(t,X_t)$. Moreover there
exists a constant $C_4$ depending on the Lipschitz constants $C_1$
and $C_2$, the bound $M$ of the terminal data, the dimension $n$ and
$d$, and the terminal time $T$ such that
$$|\Phi(t,x)|,\ |\nabla_x\Phi(t,x)|\leq C_4,\ \  \ \text{for}\ (t,x)\in[0,T]\times R^d.$$
\end{lemma}

Based on such constant $C_4$, we make a partition of $[0,T]$ by
$\pi:0=t_0\leq t_1\leq\cdots\leq t_N=T$ with the mesh
$|\pi|=\max_{1\leq i\leq N}|t_i-t_{i-1}|$ such that
$$\sqrt{|\pi|}=\frac{1}{8C_1(1+C_4)}\wedge 1.$$

We start with $[t_{N-1},t_N]$ and consider
$\Psi(N)=(V(N),X(N)^{t_{N-1},x})^{*}$ such that
$$\Psi(N)_t=\chi_x+\int_{t_{N-1}}^tF(s,Y(N)_s,Z(N)_s)ds+\int_{t_{N-1}}^t\chi_1dB_s$$
with
\begin{align*}
Y(N)_t=&\ E[\phi(X(N)^{t_{N-1},x}_T)+V(N)_T|\mathcal{F}_t]-V(N)_t,\\
\int_{t_{N-1}}^TZ(N)_sdB_s=&\ \phi(X(N)^{t_{N-1},x}_T)+V(N)_T\\
&-E[\phi(X(N)^{t_{N-1},x}_T)+V(N)_T|\mathcal{F}_{t_{N-1}}],
\end{align*}
where we used the superscripts $(t_{N-1},x)$ to emphasize
$X(N)^{t_{N-1},x}$ starting from $X(N)_{t_{N-1}}^{t_{N-1},x}=x$. By
Lemma \ref{lemmaforlocal}, there exists a unique solution:
$$\Psi(N)=(V(N),X(N)^{t_{N-1},x})^{*}\in\mathcal{C}([t_{N-1},T];R^{n+d})$$
and we also get $(Y(N),Z(N))$. Moreover, there exists a
Borel-measurable function $\Phi_{N-1}$ such that
$Y(N)_{t_{N-1}}=\Phi_{N-1}(t_{N-1},x)$ and by Lemma \ref{Delarue},
$$|\nabla_x\Phi_{N-1}(t_{N-1},x)|\leq C_4\ \  \ \text{for}\ x\in R^d.$$

In general on $[t_{i-1},t_i]$ for $1\leq i\leq N-1$, consider
$\Psi(i)=(V(i),X(i)^{t_{i-1},x})^{*}$ such that
$$\Psi(i)_t=\chi_x+\int_{t_{i-1}}^tF(s,Y(i)_s,Z(i)_s)ds+\int_{t_{i-1}}^t\chi_1dB_s$$
with
\begin{align*}
Y(i)_t=&\ E[\Phi_{i}(t_i,X(i)^{t_{i-1},x}_{t_i})+V(i)_{t_i}|\mathcal{F}_t]-V(i)_t,\\
\int_{t_{i-1}}^{t_i}Z(i)_sdB_s=&\ \Phi_i(t_i,X(i)^{t_{i-1},x}_{t_i})+V(i)_{t_i}\\
&-E[\Phi_i(t_i,X(i)^{t_{i-1},x}_{t_i})+V(i)_{t_i}|\mathcal{F}_{t_{i-1}}].
\end{align*}
By Lemma \ref{lemmaforlocal} again, there exists a unique solution:
$$\Psi(i)=(V(i),X(i)^{t_{i-1},x})^{*}\in\mathcal{C}([t_{i-1},t_i];R^{n+d})$$
and we get $(Y(i),Z(i))$ as well. Moreover there exists a
Borel-measurable function $\Phi_{i-1}$ such that
$Y(i)_{t_{i-1}}=\Phi_{i-1}(t_{i-1},x)$ and by Lemma \ref{Delarue},
$$|\nabla_x\Phi_{i-1}(t_{i-1},x)|\leq C_4\ \  \ \text{for}\ x\in R^d.$$

Of course $(V(i),X(i)^{t_{i-1},x})$ for $1\leq i\leq N$ are not the
real solutions to (\ref{MainFDE2}) on the corresponding time
interval $[t_{i-1},t_i]$, because they start from
$$(V(i)_{t_{i-1}},X(i)^{t_{i-1},x}_{t_{i-1}})=(0,x).$$
We need to shift the paths of $(V(i),X(i)^{t_{i-1},x})$ accordingly
in order to match the starting points for the solutions to
(\ref{MainFDE2}) on each time interval $[t_{i-1},t_i]$.

\begin{lemma} If the coefficients satisfy Condition \ref{Cond2.4},
then the global solution $\Psi=(V,X)^{*}$ to (\ref{MainFDE2}) is
constructed as follows: for $1\leq i\leq N$,
$$V_t=V(i)_t+\sum_{j=1}^{i-1}V(j)_{t_j}\ \ \ \text{for}\ t_{i-1}\leq t\leq t_i,$$
where we follow the convention $\sum_{j=1}^0=0$ and
\begin{eqnarray*}
    X_{t}=\left\{
    \begin{array}{ll}
    X(1)_t^{t_0,x}& \ \text{for}\ t_0\leq t\leq t_1;\\
    X(2)_t^{t_1,X_{t_1}}& \ \text{for}\ t_1\leq t\leq t_2;\\
    \cdots\\
    X(N)_t^{t_{N-1},X_{t_{N-1}}}& \ \text{for}\ t_{N-1}\leq t\leq
    t_N
    \end{array}\right.
\end{eqnarray*}
with $(Y,Z)$ being constructed as $(Y_t,Z_t)=(Y(i)_t,Z(i)_t)$ for
$t_{i-1}\leq t\leq t_i$.
\end{lemma}
\begin{proof} We only need to show $\Psi=(V,X)^{*}$ with $(Y,Z)$
satisfying (\ref{MainFDE2}) for $\tau=0$. In fact for
$t\in[t_{N-1},t_N]$, by the definition of $(V,X)$,
\begin{align*}
V_t-V_{t_{N-1}}&=V(N)_t+\sum_{j=1}^{N-1}V(j)_{t_j}-V(N)_{t_{N-1}}-\sum_{j=1}^{N-1}V(j)_{t_j}\\
&=\int_{t_{N-1}}^th(s,Y(N)_s,Z(N)_s)ds,
\end{align*}
and
\begin{align*}
X_t-X_{t_{N-1}}&=X(N)_t^{t_{N-1},X_{t_{N-1}}}-X_{t_{N-1}}\\
&=\int_{t_{N-1}}^tf(s,Y(N)_s,Z(N)_s)ds+\int_{t_{N-1}}^{t}dB_s,
\end{align*}
so
$$\Psi_t-\Psi_{t_{N-1}}=\int_{t_{N-1}}^tF(s,Y(N)_s,Z(N_s))ds+\int_{t_{N-1}}^t\chi_1dB_s,$$
where
\begin{align*}
Y(N)_t=&\
E\left[\phi\left(X(N)^{t_{N-1},X_{t_{N-1}}}_T\right)+V(N)_T|\mathcal{F}_t\right]-V(N)_t\\
=&\
E\left[\phi(X_T)+V(N)_T+\sum_{j=1}^{N-1}V(j)_{t_{j}}|\mathcal{F}_t\right]\\
&\ -V(N)_t-\sum_{j=1}^{N-1}V(j)_{t_{j}}\\
=&\ E[\phi(X_T)+V_T|\mathcal{F}_t]-V_t,
\end{align*}
and
\begin{align*}
\int_{t_{N-1}}^TZ(N)_sdB_s=&\ \phi(X(N)^{t_{N-1},X_{t_{N-1}}}_T)+V(N)_T\\
&-E\left[\phi\left(X(N)^{t_{N-1},X_{t_{N-1}}}_T\right)+V(N)_T|\mathcal{F}_{t_{N-1}}\right]\\
=&\ \phi(X_T)+V(N)_T+\sum_{j=1}^{N-1}V(j)_{t_{j}}\\
&\
-E\left[\phi(X_T)+V(N)_T+\sum_{j=1}^{N-1}V(j)_{t_{j}}|\mathcal{F}_{t_{N-1}}\right]\\
=&\ \phi(X_T)+V_T-E[\phi(X_T)+V_T|\mathcal{F}_{t_{N-1}}].
\end{align*}
Hence $(V,X)$ with $(Y,Z)$ defined in the lemma satisfy
(\ref{MainFDE2}) on $[t_{N-1},t_N]$.

In general for $1\leq i\leq N-1$, by the backward induction, it is
easy to verify $\Psi_t=(V_t,X_t)^*$ with $(Y_t,Z_t)$ also satisfy
(\ref{MainFDE2}) for $t\in[t_{i-1},t_i]$.
\end{proof}

\subsection{Uniform integrability of stochastic exponential}
In this subsection we will verify the Dol\'eans-Dade exponential
$\mathscr{E}(N)$ is a uniform-integrable martingale. To prove this
we need an appropriate martingale space. It turns out the martingale
space we need is \emph{bounded-mean-oscillation} ($BMO$)-martingale
space. As the $BMO$-martingale theory is already quite standard in
the quadratic BSDE study (e.g. \cite{MR2441926} and
\cite{MR2152241}), we only recall some basic facts that are
necessary in the following. For the further details and proofs, we
refer to He et al \cite{HWY}.

Let $M$ be a continuous local martingale on $[0,T]$. For $p\geq 1$,
define the martingale space $\mathcal{H}^p$ equipped with the norm
$||M||_{\mathcal{H}^p}=E\{[M,M]_T^{p/2}\}^{1/p}$. For $p>1$,
$\mathcal{H}^p$ is the dual space of $\mathcal{H}^q$ with $q$ being
the conjugate of $p$, i.e. $1/p+1/q=1$. However for $p=1$, the dual
space of $\mathcal{H}^1$ is strictly larger than
$\mathcal{H}^{\infty}$, the class of all martingales with bounded
quadratic variation. By the Fefferman's inequality, the dual of
$\mathcal{H}^1$ is in fact $BMO_2$, which is the subspace of
$\mathcal{H}^2$ and such that there exists a constant $C_5$,
$$E\left\{|M_T-M_{\tau}|^2|\mathcal{F}_{\tau}\right\}\leq C_5^2$$
for any stopping time $\tau\leq T$, and $C_5$ is defined to be the
$BMO_2$-norm. Similarly we can also define $BMO_p$-space for any
$p\geq 1$, which are equivalent to each other. In fact for $p\geq 1$
and $M\in\mathcal{H}^2$, there exists a constant $C_6(p)$ depending
on $p$ such that
$$||M||_{BMO_1}\leq ||M||_{BMO_p}\leq C_6(p)||M||_{BMO_1}$$
by the Jensen's inequality and the John-Nirenberg inequality
respectively. So from now on we will simply write $BMO$ without
specifying $p$.

\begin{lemma}\label{Lemma2}
If the coefficients satisfy Condition \ref{Cond2.4}, then
$$N=-\int_0^{\cdot}\langle f(s,Y_s,Z_s), dB_s\rangle_d$$
is a $BMO$-martingale under $\mathbf{Q}$, and therefore the
Dol\'eans-Dade exponential $\mathscr{E}(N)$ is a uniform-integrable
martingale under $\mathbf{Q}$.
\end{lemma}

\begin{proof}
For any stopping time $\tau\leq T$, by It\^o's isometry and the
linear growth condition $|f(t,y,z)|\leq C_1(t+|y|+|z|)$, we obtain
\begin{align}\label{Equ1}
&\ \sup_{\tau}E\left[|N_T-N_{\tau}|^2|\mathcal{F}_{\tau}\right]\nonumber\\
=&\ \sup_{\tau}E\left[\int_{\tau}^T|f(t,Y_t,Z_t)|^2dt\left|\right.\mathcal{F}_{\tau}\right]\nonumber\\
\leq&\
C_1^2T^3+3C_1^2\sup_{\tau}E\left[\int_{\tau}^T|Y_t|^2dt\left|\right.\mathcal{F}_{\tau}\right]+3C_1^2\sup_{\tau}E\left[\int_{\tau}^T|Z_t|^2dt\left|\right.\mathcal{F}_{\tau}\right],\
\ \ a.s..
\end{align}
Since $Y$ is uniformly bounded, we only need to control the last
term of (\ref{Equ1}). By applying It\^o's formula to $(Y_t)^2$ from
$\tau$ to $T$ and taking conditional expectation on
$\mathcal{F}_{\tau}$, we obtain
\begin{align*}
&\ |Y_{\tau}|^2+E\left[\int_{\tau}^T|Z_t|^2dt\left|\right.\mathcal{F}_{\tau}\right]\\
=&\ E[\phi(X_T)^2|\mathcal{F}_{\tau}]+2E\left[\int_{\tau}^T\langle
Y_t,h(t,Y_t,Z_t)\rangle_{n}dt\left|\right.\mathcal{F}_{\tau}\right]\\
\leq&\
M^2+\lambda^2E\left[\int_{\tau}^T|Y_t|^2dt\left|\right.\mathcal{F}_{\tau}\right]+
\frac{1}{\lambda^2}E\left[\int_{\tau}^T|h(t,Y_t,Z_t)|^2dt\left|\right.\mathcal{F}_{\tau}\right]\\
\leq&\
M^2+\lambda^2E\left[\int_{\tau}^T|Y_t|^2dt\left|\right.\mathcal{F}_{\tau}\right]+
\frac{C_1^2T^3}{\lambda^2}+\frac{3C_1^2}{\lambda^2}E\left[\int_{\tau}^T|Y_t|^2dt\left|\right.\mathcal{F}_{\tau}\right]\\
&\
+\frac{3C_1^2}{\lambda^2}E\left[\int_{\tau}^T|Z_t|^2dt\left|\right.\mathcal{F}_{\tau}\right],\
\  \ a.s.,
\end{align*}
where we used the elementary inequality
$2ab\leq\lambda^2+b^2/\lambda^2$. By choosing $\lambda$ large enough
such that $1-3C_1^2/\lambda^2>0$, and by the uniform boundedness of
$Y$, there exists a constant $C_7$ such that
$$\sup_{\tau}E\left[\int_{\tau}^T|Z_t|^2dt\left|\right.\mathcal{F}_{\tau}\right]\leq C_7,\ \ \ a.s.,$$
and the conclusion follows by plugging the above estimate into
(\ref{Equ1}).
\end{proof}

\section{Uniqueness and strong solutions for $n=1$}
As in the classical SDE theory, there are also several notions of
uniqueness for BSDEs as well. In this section we discuss the
uniqueness of BSDE (\ref{MainqBSDE1}). We assume the following
condition on the coefficients:

\begin{condition}\label{Cond2.5}
Condition \ref{Cond2.4} is assumed to be satisfied. Moreover $n=1$,
i.e. BSDE (\ref{MainqBSDE1}) is a scalar BSDE; $F=F(t,z)$ with
$F=(h,f)^{*}$, i.e.  both of the coefficients $h$ and $f$ only
depend on $t$ and $z$; and $f^j=f^j(t,z^j)$ for $j=1,\cdots,d$, i.e.
there is no mixture terms of $z$ in $f$.
\end{condition}

The weak solution can be regarded as probability distribution on the
sample path space, so we will specify the sample path space of
(\ref{MainqBSDE1}) firstly. By $\mathbf{W}^m$ we denote the space of
continuous functions $C([0,T];R^m)$. Define the coordinate mapping
$X_t:\mathbf{W}^m\rightarrow R^m$ by
$$X_t(x)=x_t,\ \ \ \text{for}\ x\in \mathbf{W}^m,$$
and on $\mathbf{W}^m$, define the following $\sigma$-algebras:
$$\mathcal{B}^X_t=\sigma(x_s:s\leq t);\ \ \ \mathcal{B}^X_{\hat{t}}=\sigma(x_u-x_t:t\leq u\leq T);$$
and $\mathcal{B}^X=\vee_{t\in[0,T]}\mathcal{B}^X_t$. Obviously we
have the relationship
$\mathcal{B}^X=\mathcal{B}^X_t\vee\mathcal{B}^X_{\hat{t}}$ for any
$t\in[0,T]$. By Definition \ref{Def}, $Y$ and $W$ must be
continuous. However it is not obvious at all that the density
representation $Z$ has any path regularity. Fortunately under
Condition \ref{Cond2.5}, Imkeller and Dos Reis \cite{Imkeller}
already did this work for us, which states that there is a
continuous modification of $Z$. We will choose such continuous
version of $Z$ from now on.

If $(\Omega,\mathcal{F},\mathbf{P})$, $\{\mathcal{F}_t\}$, and
$(Y,Z,W)$ is one weak solution of (\ref{MainqBSDE1}), we can
consider the image measure of $\mathbf{P}$ under the mapping
$(Y,Z,W):\Omega\rightarrow\mathbf{W}^{1+d}\times\mathbf{W}^d$
defined by
$$\bar{\omega}\mapsto\left((Y(\bar{\omega}),Z(\bar{\omega})),W(\bar{\omega})\right),\ \ \ \text{for}\ \bar{\omega}\in\Omega,$$
which is denoted by $\mathbf{Q}$. Since the projection of
$\mathbf{Q}$ on the third component $W$ is a Wiener measure on
$(\mathbf{W}^d,\mathcal{B}^W)$, denoted by $\mathbf{Q}^*$ from now
on, and all the spaces are Polish under the uniform topology, there
exists a unique regular conditional probability
$\mathbf{Q}\{\cdot|\omega\}$ such that:

(1) for $\omega\in\mathcal{B}^W$, $\mathbf{Q}\{\cdot|\omega\}$ is a
probability measure on
$(\mathbf{W}^{1+d},\mathcal{B}^{Y}\otimes\mathcal{B}^{Z})$;

(2) for $A\in\mathcal{B}^{Y}\otimes\mathcal{B}^{Z}$, the map
$\omega\mapsto\mathbf{Q}\{A|\omega\}$ is
$\mathcal{B}^{W}$-measurable;

(3) for  $A\in\mathcal{B}^{Y}\otimes\mathcal{B}^{Z}$ and
$B\in\mathcal{B}^W$, we have
$$\mathbf{Q}(A\times B)=\int_{B}\mathbf{Q}\{A|\omega\}\mathbf{Q}^*(d\omega).$$

\begin{definition}
The weak solution to (\ref{MainqBSDE1}) is called unique in law if
for any two weak solutions $(\Omega,\mathcal{F},\mathbf{P})$,
$\{\mathcal{F}_t\}$, $(Y,Z,W)$ and
$(\bar{\Omega},\bar{\mathcal{F}},\bar{\mathbf{P}})$,
$\{\bar{\mathcal{F}}_t\}$, $(\bar{Y},\bar{Z},\bar{W})$, the
probability distributions of $(Y,Z)$ and $(\bar{Y},\bar{Z})$ are
equal. i.e.
$\mathbf{P}_{(Y,Z)}=\bar{\mathbf{P}}_{(\bar{Y},\bar{Z})}$.

The weak solution to (\ref{MainqBSDE1}) is called pathwise unique if
for any two weak solutions $(Y,Z)$ and $(\bar{Y},\bar{Z})$ defined
on the same probability space $(\Omega,\mathcal{F},\mathbf{P})$ with
the filtration $\{\mathcal{F}_t\}$ and the same Brownian motion $W$,
$(Y,Z)$ is a continuous modification of $(\bar{Y},\bar{Z})$.
\end{definition}

For a SDE, the celebrated Yamada-Watanabe theorem states that the
weak existence and pathwise uniqueness of the solutions to a SDE
implies the existence of a strong solution. As Kurtz
\cite{MR2336594} pointed out: strong solution is a consequence of
measurable selection, and such result has little to do with the
equation, but really a consequence of the convexity of collections
of the probability distributions of solutions. If the compatibility
constraint (See Definition \ref{Def}) is satisfied, we further have
the adapteness of solutions.

\begin{theorem}\label{maintheorem2} If the coefficients satisfy Condition \ref{Cond2.5},
then the weak solution to BSDE (\ref{MainqBSDE1}) is pathwise
unique, and the strong solution also exists.
\end{theorem}

We first prove the pathwise uniqueness. By employing the Girsanov's
theorem reversely, we have the following pathwise uniqueness result.

\begin{lemma} If the coefficients satisfy Condition \ref{Cond2.5},
then the weak solution to (\ref{MainqBSDE1}) is pathwise unique.
\end{lemma}

\begin{proof} Suppose $(Y,Z)$ and $(\bar{Y},\bar{Z})$ are two weak solutions
on $(\Omega,\mathcal{F},\mathbf{P})$ with $\{\mathcal{F}_t\}$ and
Brownian motion $W$. By applying It\^o's formula to $e^{\alpha
t}(Y_t-\bar{Y}_t)^{2}$ for some $\alpha$ to be determined, we obtain
\begin{align}\label{chap4:Equ2222}
&e^{\alpha t}(Y_t-\bar{Y}_t)^2\nonumber\\
=&-2\int_t^Te^{\alpha
s}(Y_s-\bar{Y}_s)d(Y_s-\bar{Y}_s)-\int_t^Te^{\alpha
s}d[Y-\bar{Y},Y-\bar{Y}]_s\nonumber\\
&-\int_t^T\alpha e^{\alpha
s}(Y_s-\bar{Y}_s)^2ds\nonumber\\
=&\
2\int_t^Te^{\alpha s}(Y_s-\bar{Y}_s)\left\{\sum_{j=1}^d(Z^j_sf^j(s,Z^j_s)-\bar{Z}^j_sf^j(s,\bar{Z}^j_s))+(h(s,Z_s)-h(s,\bar{Z}_s))\right\}ds\nonumber\\
&-2\int_t^Te^{\alpha
s}(Y_s-\bar{Y}_s)\sum_{j=1}^d(Z^j_s-\bar{Z}^j_s)dW_s^j-\int_t^Te^{\alpha
s}\sum_{j=1}^d|Z_s^j-\bar{Z}_s^j|^2ds\nonumber\\
&-\int_t^T\alpha e^{\alpha s}(Y_s-\bar{Y}_s)^2ds.
\end{align}

Note that for $s\in[0,T]$, and $z^j,\bar{z}^j\in R$ for
$j=1,\cdots,d$,
\begin{align*}
&\left|z^jf^j(s,z^j)-\bar{z}^jf^j(s,\bar{z}^j)\right|\\
\leq&\ \left|z^jf^j(s,z^j)-\bar{z}^jf^j(s,z^j)\right|
+\left|\bar{z}^jf^j(s,z^j)-\bar{z}^jf^j(s,\bar{z}^j)\right|\\
\leq&\
C_1(T+|z^j|)|z^j-\bar{z}^j|+C_1|\bar{z}^j||z^j-\bar{z}^j|\\
\leq&\ C_{1}(T+|z^j|+|\bar{z}^j|)|z^j-\bar{z}^j|.
\end{align*}
Now if we set
$$\beta_s^j=\frac{Z_s^jf^j(s,Z_s^j)-\bar{Z}_s^jf^j(s,\bar{Z}_s^j)}{Z_s^j-\bar{Z}_s^j},\ \ \ \text{for}\ s\in[0,T],$$
when $Z_s^j-\bar{Z}_s^j\neq 0$, and $\beta_s^j=0$ for $s\in[0,T]$
otherwise, then there exists a constant $C_{8}$ such that
$|\beta_s^j|^2\leq C_{8}(1+|Z_{s}^j|^2+|\bar{Z}_{s}^j|^2)$. Based on
such $\beta^j$, we define a new probability measure $\mathbf{Q}$ by
$\frac{d\mathbf{Q}}{d\mathbf{P}}=\mathscr{E}(N)$ where
$$N=\sum_{j=1}^d\int_0^{\cdot}\beta_s^jdW_s^j,$$
and under $\mathbf{Q}$ define a new Brownian motion $B$ by
$B=W-[W,N]$. Then under the probability measure $\mathbf{Q}$,
(\ref{chap4:Equ2222}) reduces to
\begin{align*}
&e^{\alpha t}(Y_t-\bar{Y}_t)^2\\
=&-2\int_t^Te^{\alpha
s}(Y_s-\bar{Y}_s)\sum_{j=1}^d(Z^j_s-\bar{Z}^j_s)dB_s^j+2\int_t^Te^{\alpha
s}(Y_s-\bar{Y}_s)(h(s,Z_s)-h(s,\bar{Z}_s))ds\\
&-\int_t^Te^{\alpha
s}\sum_{j=1}^d|Z_s^j-\bar{Z}_s^j|^2ds-\int_t^T\alpha e^{\alpha
s}(Y_s-\bar{Y}_s)^2ds.
\end{align*}
By taking expectation under $\mathbf{Q}$ we have
\begin{align*}
E^{\mathbf{Q}}[e^{\alpha t}(Y_t-\bar{Y}_t)^2]=&\
E^{\mathbf{Q}}\left\{\int_t^Te^{\alpha
s}2(Y_s-\bar{Y}_s)(h(s,Z_s)-h(s,\bar{Z}_s))ds\right\}\\
&-E^{\mathbf{Q}}\left\{\int_t^Te^{\alpha
s}|Z_s-\bar{Z}_s|^2ds\right\} -E^{\mathbf{Q}}\left\{\int_t^T\alpha
e^{\alpha s}(Y_s-\bar{Y}_s)^2ds\right\}.
\end{align*}
By the elementary inequality $2ab\leq \lambda^2a^2+b^2/\lambda^2$,
$$
2(Y_s-\bar{Y}_s)(h(s,Z_s)-h(s,\bar{Z}_s))\leq
\lambda^2(Y_s-\bar{Y}_s)^2+\frac{C_1^2}{\lambda^2}|Z_s-\bar{Z}_s|^2.
$$
By choosing $\lambda^2=\alpha$ and $\alpha=2C_1^2$, we obtain
$$E^{\mathbf{Q}}[e^{2C_1^2 t}(Y_t-\bar{Y}_t)^2]
\leq-\frac12E^{\mathbf{Q}}\left\{\int_t^Te^{2C_1^2s}|Z_s-\bar{Z}_s|^2ds\right\}\leq
0,$$ so $Y_t=\bar{Y}_t$ for $t\in[0,T],$ $a.s.$, and $Z_t=\bar{Z}_t$
for $a.e.\ t\in[0,T]$, $a.s..$ Now the only step left is to verify
$\mathscr{E}(N)$ is a uniformly-integrable martingale, and we need
to verify $N$ is a $BMO$-martingale under $\mathbf{P}$. In fact for
any stopping time $\tau\leq T$,
\begin{align*}
E\left\{|N_T-N_{\tau}|^2|\mathcal{F}_{\tau}\right\}&
=\sum_{j=1}^dE\left\{\int_{\tau}^T|\beta_s^j|^2ds|\mathcal{F}_{\tau}\right\}\\
&\leq C_{8}
E\left\{\int_{\tau}^T(d+|Z_s|^2+|\bar{Z}_s|^2)ds|\mathcal{F}_{\tau}\right\},\
\ \ a.s..
\end{align*}
The way to control the integral term involving $Z$ and $\bar{Z}$ has
already been presented in the proof of Lemma \ref{Lemma2}. Therefore
$\mathbf{Q}$ defined above is indeed a probability measure.
\end{proof}

The idea for the following lemma is standard: to transfer the
structure of a weak solution such that $\mathbf{W}^{1+d}$ becomes
the sample path space for $(Y,Z)$ and $\mathbf{W}^d$ that for $W$.
What allows us to carry it through is the regular conditional
probability introduced above.

\begin{lemma} For BSDE (\ref{MainqBSDE1}) with the coefficients satisfying Condition \ref{Cond2.5}, pathwise uniqueness implies uniqueness in
law, and furthermore if the weak solution exists, then the strong
solution also exists.
\end{lemma}
\begin{proof}Since the proof is quite standard, we only present the basic steps.

Let $(\Omega,\mathcal{F},\mathbf{P})$, $\{\mathcal{F}_t\}$,
$(Y,Z,W)$ and $(\bar{\Omega},\bar{\mathcal{F}},\bar{\mathbf{P}})$,
$\{\bar{\mathcal{F}}_t\}$, $(\bar{Y},\bar{Z},\bar{W})$ be two weak
solutions. Let $\mathbf{W}^{1+d}$ and $\bar{\mathbf{W}}^{1+d}$ be
two copies of $C([0,T];R)\times C([0,T];R^d)$. By using the regular
conditional probability $\mathbf{Q}\{\cdot|\omega\}$ and
$\bar{\mathbf{Q}}\{\cdot|\omega\}$, we define a probability measure
$\pi$ on the probability space $(\Theta,\mathcal{B}(\Theta))$ by
$$\pi\left((dy,dz),(d\bar{y},d\bar{z}),d\omega\right)=\mathbf{Q}\{(dy,dz)|\omega\}
\bar{\mathbf{Q}}\{(d\bar{y},d\bar{z})|\omega\}\mathbf{Q}^*(d\omega)$$
where
$$(\Theta,\mathcal{B}(\Theta))=
(\mathbf{W}^{1+d}\times\bar{\mathbf{W}}^{1+d}\times\mathbf{W}^d,
\mathcal{B}^Y\otimes\mathcal{B}^Z\otimes\mathcal{B}^{\bar{Y}}\otimes\mathcal{B}^{\bar{Z}}
\otimes\mathcal{B}^W).$$ On $(\Theta,\mathcal{B}(\Theta),\pi)$, we
further define the filtration $\{\mathcal{G}_t\}$ which is generated
by $\sigma((y_s,z_s),(\bar{y}_s,\bar{z}_s),\omega_s:s\leq t)$
augmented by the $\pi$-null sets in $\mathcal{B}(\Theta)$. Then
under $\pi$ and $\{\mathcal{G}_t\}$, $\omega$ is still a Brownian
motion. In fact by the compatibility constraint in Definition 3,
$\mathcal{B}^{Y}_t\otimes\mathcal{B}_t^{Z}$ is independent of
$\mathcal{B}^W_{\hat{t}}$. Hence for
$A_t\in\mathcal{B}_t^Y\otimes\mathcal{B}_t^Z$,
$$\mathbf{Q}\{A_t|\omega\}=\mathbf{Q}\{A_t|\omega_t\}.$$
Likewise we also have
$\bar{\mathbf{Q}}\{\bar{A}_t|\omega\}=\bar{\mathbf{Q}}\{\bar{A}_t|\omega_t\}$
for
$\bar{A}_t\in\mathcal{B}_t^{\bar{Y}}\otimes\mathcal{B}_t^{\bar{Z}}$.
Based on above relationships, for $B_t\in\mathcal{B}_t^{W}$,
$u\in[t,T]$ and $\xi\in R^d$,
\begin{align*}
&E^{\pi}\left[e^{i\langle\xi,\omega_u-\omega_t\rangle_d}\emph{1}_{A_t}\emph{1}_{\bar{A}_t}\emph{1}_{B_t}\right]\\
=&\int_{B_t}e^{i\langle\xi,\omega_u-\omega_t\rangle_d}
\mathbf{Q}\{A_t|\omega\}\bar{\mathbf{Q}}\{\bar{A}_t|\omega\}\mathbf{Q}^*(d\omega)\\
=&\int_{\mathbf{W}^d}e^{i\langle\xi,\omega_u-\omega_t\rangle_d}\mathbf{Q}^{*}(d\omega)\int_{B_t}
\mathbf{Q}\{A_t|\omega_t\}\bar{\mathbf{Q}}\{\bar{A}_t|\omega_t\}\mathbf{Q}^*(d\omega)\\
=&E^{\pi}\left[e^{i\langle\xi,\omega_u-\omega_t\rangle_d}\right]E^{\pi}\left[\emph{1}_{A_t}\emph{1}_{\bar{A}_t}\emph{1}_{B_t}\right],
\end{align*}
which means $\{\omega_u-\omega_t:t\leq u\leq T\}$ is independent of
the filtration $\{\mathcal{G}_t\}$.

Therefore $(y,z,\omega)$ and $(\bar{y},\bar{z},\omega)$ are two weak
solutions on the same filtered probability space
$(\Theta,\mathcal{B}(\Theta),\{\mathcal{G}_t\},\pi)$. Pathwise
uniqueness means
\begin{equation}\label{Yamada1}
\pi\left(\{(y,z),(\bar{y},\bar{z}),\omega\}\in\Theta:
(y,z)=(\bar{y},\bar{z})\right)=1,
\end{equation} so for any
$A\in\mathcal{B}^Y\otimes\mathcal{B}^Z$, the probability
distribution $\mathbf{P}_{(Y,Z)}(\bar{\omega}\in\Omega: (Y,Z)\in A)$
equals
$$
\pi\left(\{(y,z),(\bar{y},\bar{z}),\omega\}\in\Theta: (y,z)\in
A\right) =\pi\left(\{(y,z),(\bar{y},\bar{z}),\omega\}\in\Theta:
(\bar{y},\bar{z})\in A\right)
$$
which is equal to the probability distribution
$\bar{\mathbf{P}}_{(\bar{Y},\bar{Z})}(\bar{\omega}\in\bar{\Omega}:
(\bar{Y},\bar{Z})\in A)$.

To prove the second claim, we firstly show
$\mathbf{Q}\{\cdot|\omega\}$ and $\bar{\mathbf{Q}}\{\cdot|\omega\}$
assign full measure to the same singleton. By (\ref{Yamada1}) and
the definition of $\pi$, we have
$$\int_{\mathbf{W}^{d}}\int_{(y,z)=(\bar{y},\bar{z})}\mathbf{Q}\{(dy,dz)|\omega\}
\bar{\mathbf{Q}}\{(\bar{y},\bar{z})|\omega\}\mathbf{Q}^*(d\omega)=1,$$
so there exists $N\in\mathcal{B}^W$ with $\mathbf{Q}^*(N)=0$ such
that
$$\int_{(y,z)=(\bar{y},\bar{z})}\mathbf{Q}\{(dy,dz)|\omega\}
\bar{\mathbf{Q}}\{(\bar{y},\bar{z})|\omega\}=1,\ \ \ \text{for}\
\omega\in N^c.$$ This can only occur if there exists a
$\mathcal{B}^W/\mathcal{B}^Y\otimes\mathcal{B}^Z$-measurable map
$\Phi=(\Phi^Y,\Phi^Z): \mathbf{W}^{d}\rightarrow\mathbf{W}^{1+d}$
such that
$$\mathbf{Q}\{(y,z)|\omega\}=\bar{\mathbf{Q}}\{(y,z)|\omega\}=
\delta_{\Phi(\omega)}(y,z),\ \ \ \text{for}\ \omega\in N^c.$$ It
then follows $(y,z)=(\Phi^Y(\omega),\Phi^Z(\omega))$ for $\omega\in
N^c$. But recalling $\omega\mapsto\mathbf{Q}\{A_t|\omega\}$ is
$\mathcal{B}^{W}_t$-measurable for
$A_t\in\mathcal{B}^Y_t\otimes\mathcal{B}^Z_t$, then by the standard
Dynkin arguments, $\Phi$ is in fact also
$\mathcal{B}^W_t/\mathcal{B}^Y_t\otimes\mathcal{B}^Z_t$-measurable,
and on any given filtered probability space
$(\Omega,\mathcal{F},\mathcal{F}_t,\mathbf{P})$ satisfying the usual
conditions with $W$ being a Brownian motion on it,
\begin{align*}
\Phi^Y(W_t)=&\ \phi(W_T)+\int_t^Th(s,\Phi^Z(W_s))ds\\
&\ + \int_t^T\sum_{j=1}^d\Phi^{Z,j}(W_s)f^j(s,\Phi^{Z,j}(W_s))ds
-\int_t^T\sum_{j=1}^d\Phi^{Z,j}(W_s)dW_s^j,
\end{align*}
so $\Phi$ is a strong solution.
\end{proof}

\section{Optimal portfolio in incomplete markets}

In this section our main aim is to demonstrate how the functional
differential equation approach and the nonlinear Girsanov's
transformation can be used in finance. Specifically we consider an
example of optimal portfolio problems in incomplete markets which is
often used in the indifference valuation.

\begin{assumption} (the probability space)
Let $(\Omega,\mathcal{F},\mathbf{P})$ be a complete probability
space which is to be determined, and $\{\mathcal{F}_t\}$ be its
associated filtration satisfying the usual conditions, which is also
to be determined.
\end{assumption}

\begin{assumption} (the market)
The market is built with three assets: a risk-free bond with zero
interest rate, a tradeable asset and a nontradeable asset. The
pricing dynamic of the tradeable asset satisfies the following SDE
on the above given probability space:
\begin{equation}\label{chap5:dynforS}
\left\{
\begin{array}{ll}
dS_t/S_t=\mu_t^Sdt+\bar{\sigma}_t^Sd\bar{W}_t,\\
S_0=s,
\end{array}
\right.
\end{equation}
and the dynamic of the nontradeable asset follows
\begin{equation}\label{chap5:dynforV}
\left\{
\begin{array}{ll}
dV_t/V_t=\mu_t^Vdt+\sigma_t^VdW_t^V+\bar{\sigma}_t^Vd\bar{W}_t,\\
V_0=v
\end{array}
\right.
\end{equation}
on the same given probability space, where
$\mathbf{W}=(W^V,\bar{W})$ is a two-dimensional Brownian motion to
be determined. $\sigma^V$ and $\mu^i,\bar{\sigma}^i$ for $i=S,V$, as
the market coefficients, are bounded and continuous functions.
\end{assumption}

\begin{assumption} (the investor)
The investor has an exponential utility function depending on
his/her terminal wealth, which has the form:
$$U(x)=-e^{-\gamma x},\ \ \ \text{for}\ x\in R,$$
where $\gamma\geq 0$ representing the degree of the investor's risk
aversion.
\end{assumption}

\begin{assumption} (the trading strategy)
The investor, with initial wealth $x$, invests in the tradeable
asset and the risk-free bond during the time period $[0,T]$. Let
$\pi$ be the amount of money invested in the tradeable asset. We
assume $\pi$ is taken from the following admissible set, which of
course depends on the above given probability space.
$$\mathcal{A}_{ad}:=\{\pi:[0,T]\times\Omega\rightarrow R:\ \pi\ \text{is}\
\mathcal{F}_t\text{-adapted},\ \text{self-financing and}\
||\pi||_{H^2[0,T]}<\infty.\}$$ The dynamic of the investor's wealth
process, denoted by $X^{x}(\pi)$, follows
\begin{equation}\label{chap5:dynforX}
\left\{
\begin{array}{ll}
dX_t^x(\pi)=\pi_t(\mu_t^Sdt+\bar{\sigma}_t^Sd\bar{W}_t),\\[+0.1cm]
X_0^x(\pi)=x.
\end{array}
\right.
\end{equation}
\end{assumption}

\begin{assumption} (the cost functional)
At time $t=T$, the investor gets the total amount $X_T^{x}(\pi)$
plus a random endowment $g(V_T,S_T)$, where $g$ is assumed to be
Lipschitz continuous and uniform bounded. The investor decides the
optimal trading strategy to maximize the following cost functional:
$$\sup_{\pi\in\mathcal{A}_{ad}}E^{\mathbf{P}}\left[-e^{-\gamma(X_T^{x}(\pi)+g(V_T,S_T))}\right].$$
Here we use the superscript $\mathbf{P}$ to emphasize the
expectation is taken under the probability measure $\mathbf{P}$,
which is to be determined.
\end{assumption}

The random endowment $g$ depends not only on the nontradeable asset
$V$ but also on the tradeable asset $S$, which distinguishes the
current problem from the ones usually considered in the literature.
In the book edited by Carmona \cite{Carmona}, this problem is even
called an open problem (\cite{HH2009}). Such form of random
endowment actually appears naturally when one wants to consider the
credit risk of options traded in OTC markets (see Henderson and
Liang \cite{HL2010}). We also emphasize the well known Cole-Hopf
transformation does not help to deduce the closed-from solutions in
our setting.

In the following we give the definition of weak admissible trading
strategy and the corresponding weak formulation of optimal portfolio
problems. For the weak formulation of general stochastic control
problems, we refer to Yong and Zhou \cite{MR1696772}.

\begin{definition}\label{chap5:def} A triple $(\Omega,\mathcal{F},\mathbf{P})$
$\{\mathcal{F}_t\}$ and $(\pi,\mathbf{W})$ is called a weak
admissible trading strategy if

\noindent\textbf{(1)} $(\Omega,\mathcal{F},\mathbf{P})$ is a
complete probability space with the filtration $\{\mathcal{F}_t\}$
satisfying the usual conditions;

\noindent\textbf{(2)} $\mathbf{W}$ is a Brownian motion, and the
increment $\{\mathbf{W}_u-\mathbf{W}_t:t\leq u\leq T\}$ must be
independent of $\sigma$-algebra $\mathcal{F}_t$;

\noindent\textbf{(3)} $\pi$ is taken from the admissible set
$\mathcal{A}_{ad}.$

The set of all weak admissible trading strategies is denoted as
$\mathcal{A}_{ad}^W$, and a generic element in such weak admissible
set $\mathcal{A}_{ad}^W$ is denoted as $\Pi$. The investor decides
the optimal weak admissible trading strategy $\Pi$ in order to
maximize his/her cost functional:
\begin{equation}\label{chap5:valuefunction}
\sup_{\Pi\in\mathcal{A}_{ad}^W}E^{\mathbf{P}}\left[-e^{-\gamma(X_T^{x}(\pi)+g(V_T,S_T))}\right].
\end{equation}
\end{definition}

\begin{remark}
The motivation of introducing the above weak formulation of optimal
portfolio problems is more from mathematics rather than finance.
Later We will use the martingale optimality principle to deduce a
quadratic BSDE as the characterization of the optimal portfolio, and
we will look for the weak solution of such quadratic BSDE. The
probability space will be chosen from the weak solution of the
associated quadratic BSDE.
\end{remark}

Next we use the martingale optimality principle to characterize the
optimal portfolio. For a given filtered probability space
$(\Omega,\mathcal{F},\mathcal{F}_t,\mathbf{P})$ with a
two-dimensional Brownian motion $\mathbf{W}=(W^V,\bar{W})$, all of
which are to be determined, we want to construct a family of
stochastic processes
$$(-e^{-\gamma(X_t^{x}(\pi)+Y_t)})_{t\in[0,T]}$$
such that\\[-0.3cm]

\noindent\textbf{(1)} the process
$(-e^{-\gamma(X_t^{x}(\pi)+Y_t)})_{t\in[0,T]}$ is a supermartingale
for any $\pi\in\mathcal{A}_{ad}$, and there exists an optimal
$\pi^{*}\in\mathcal{A}_{ad}$ such that
$(-e^{-\gamma(X_t^{x}(\pi^*)+Y_t)})_{t\in[0,T]}$ is a martingale;

\noindent\textbf{(2)} the auxiliary process $(Y_t)_{t\in[0,T]}$ has
the
terminal value $Y_T=g(V_T,S_T)$.\\[-0.3cm]

If such auxiliary process $(Y_t)_{t\in[0,T]}$ and the optimal
$\pi^{*}$ exist, then we have
$$E^{\mathbf{P}}\left[-e^{-\gamma(X_T^{x}(\pi)+Y_T)}\right]\leq
-e^{-\gamma(x+Y_0)}, \ \ \ \text{for\ any}\
\pi\in\mathcal{A}_{ad},$$ and
$$E^{\mathbf{P}}\left[-e^{-\gamma(X_T^{x}(\pi^{*})+Y_T)}\right]=
-e^{-\gamma(x+Y_0)}, \ \ \ \text{for\ optimal}\
\pi^{*}\in\mathcal{A}_{ad}.$$ Therefore
\begin{align*}
\sup_{\pi\in\mathcal{A}_{ad}}E^{\mathbf{P}}\left[-e^{-\gamma(X_T^{x}(\pi)+Y_T)}\right]
&=E^{\mathbf{P}}\left[-e^{-\gamma(X_T^{x}(\pi^{*})+Y_T)}\right]\\
&=-e^{-\gamma(x+Y_0)}.
\end{align*}

Note that the filtered probability space
$(\Omega,\mathcal{F},\mathcal{F}_t,\mathbf{P})$ and the Brownian
motion $\mathbf{W}$ are still to be determined. Next we use the weak
solution of a quadratic BSDE to characterize the auxiliary processes
$Y$ and $\pi^*$, which also provides us with the filtered
probability space $(\Omega,\mathcal{F},\mathcal{F}_t,\mathbf{P})$
and the Brownian motion $\mathbf{W}$.

\begin{theorem}\label{Prop5.1} Let $(\Omega,\mathcal{F},\mathbf{P})$, $\{\mathcal{F}_t\}$
and $(Y,\mathbf{Z},\mathbf{W})$ (with $\mathbf{Z}=(Z^V,\bar{Z})$) be
the weak solution to the following quadratic BSDE:
\begin{equation}\label{chap5:qbse1}
Y_t=g(V_T,S_T)-\int_t^Tf_sds-\int_t^T(Z_s^VdW_s^V+\bar{Z}_sd\bar{W}_s)
\end{equation}
with
$$f_t=\frac{\gamma}{2}(Z_t^V)^2+\frac{\mu_t^S}{\bar{\sigma}_t^S}\bar{Z}_t-\frac{(\mu_t^S)^2} {2\gamma(\sigma_t^S)^2}.$$ Then the value
function of the optimal portfolio problem
(\ref{chap5:valuefunction}) is given by
$$-e^{-\gamma(x+Y_0)},$$
and the optimal weak admissible trading strategy $\Pi^*$ is the
triple $(\Omega,\mathcal{F},\mathbf{P})$, $\{\mathcal{F}_t\}$ and
$(\pi^*,\mathbf{W})$ with
\begin{equation}\label{optimaltradingpi}
\pi^{*}_t=-\bar{Z}_t+\frac{\mu_t^S} {\gamma(\bar{\sigma}_t^S)^2}.
\end{equation}
\end{theorem}

\begin{proof} On a given filtered probability space
$(\Omega,\mathcal{F},\mathcal{F}_t,\mathbf{P})$ with the Brownian
motion $\mathbf{W}$, we suppose $(Y_t)_{t\in[0,T]}$ satisfies the
following BSDE:
$$Y_t=g(V_T,S_T)-\int_t^Tf_sds-\int_t^T(Z_s^VdW_s^V+\bar{Z}_sd\bar{W}_s),$$
where the driver $f$ is to be determined. By applying It\^o's
formula to $e^{-\gamma(X_t^{x}(\pi)+Y_t)}$, we obtain
\begin{align*}
&de^{-\gamma(X_t^{x}(\pi)+Y_t)}\\
=&\ e^{-\gamma(X_t^{x}(\pi)+Y_t)}\left\{-\gamma(dX_t^{x}(\pi)+
dY_t)+\frac{\gamma^2}{2}d[X^{x}(\pi)+Y,X^{x}(\pi)+Y]_t
\right\}\\
=&\ e^{-\gamma(X_t^{x}(\pi)+Y_t)}\left\{-\gamma\mu_t^S\pi_t-\gamma
f_t+\frac{\gamma^2}{2}\left[(\bar{\sigma}_t^S)^2\pi_t^2+(Z_t^V)^2+
(\bar{Z}_t)^2\right.\right.\\
&\left.\left.+2\bar{\sigma}_t^S\bar{Z}_t\pi_t\right]\right\}dt+
\text{martingale term}.
\end{align*}
Since $(-e^{-\gamma(X_t^{x}(\pi)+Y_t)})_{t\in[0,T]}$ is
supermartingale for any $\pi\in\mathcal{A}_{ad}$, and a martingale
for optimal $\pi^{*}\in\mathcal{A}_{ad}$, we must have $f_t$ and
$\pi^*$ such that
\begin{equation*}
\frac{\gamma^2}{2}(\bar{\sigma}_t^S)^2\left(\pi_t+
\bar{Z}_t-\frac{\mu_t^S} {\gamma(\bar{\sigma}_t^S)^2}\right)^2
+\frac{\gamma^2}{2}(Z_t^V)^2+\frac{\gamma\mu_t^S}{\bar{\sigma}_t^S}\bar{Z}_t-\frac{(\mu_t^S)^2}
{2(\sigma_t^S)^2}-\gamma f_t\geq 0
\end{equation*}
for any $\pi\in\mathcal{A}_{ad}$, and equality holds for optimal
$\pi^*$. Therefore by solving the above variational inequality, we
obtain
$$f_t=\frac{\gamma}{2}(Z_t^V)^2+\frac{\mu_t^S}{\bar{\sigma}_t^S}\bar{Z}_t-\frac{(\mu_t^S)^2} {2\gamma(\sigma_t^S)^2},$$ and
$$\pi^{*}_t=-\bar{Z}_t+\frac{\mu_t^S} {\gamma(\bar{\sigma}_t^S)^2}.$$
\end{proof}

In the following we will employ the functional differential equation
approach and the nonlinear Girsanov's transformation to find the
weak solution of BSDE (\ref{chap5:qbse1}). Since the coefficients
satisfy Condition \ref{Cond2.5}, by Theorem \ref{maintheorem2}, the
weak solution we will find is pathwise unique, and moreover the
strong solution also exists.

The idea is to use the strong solution of the following FBSDE
(\ref{chap5:MainFBSDE22}) to construct the weak solution of BSDE
(\ref{chap5:qbse1}). Let's start with a Brownian motion
$\mathbf{B}=(B^V,\bar{B})$ on $(\Omega,\mathcal{F},\mathbf{Q})$ with
the filtration $\{\mathcal{F}_t\}$ satisfying the usual conditions,
and consider the following FBSDE:
\begin{eqnarray}\label{chap5:MainFBSDE22}
    \left\{
    \begin{array}{lll}
    d\ln
    V_t&=&\displaystyle\left\{\mu_t^V-\frac12[(\sigma_t^V)^2+(\bar{\sigma}_t^V)^2]\right\}dt
    -\sigma_t^V\frac{\gamma}{2}Z_t^Vdt\\[+0.5cm]
    &&-\bar{\sigma}_t^V\displaystyle\frac{\mu_t^S}{\bar{\sigma}_t^S}dt
    +\sigma_t^VdB_t^V+\bar{\sigma}_t^Vd\bar{B}_t,\\[+0.5cm]
    \ln V_0&=&\ln v,\\[+0.5cm]
    d\ln
    S_t&=&-\displaystyle\frac12(\bar{\sigma}_t^S)^2dt+\bar{\sigma}_t^Sd\bar{B}_t,\\[+0.5cm]
    \ln S_0&=&\ln s,\\[+0.5cm]
    dY_t&=&-\displaystyle\frac{(\mu_t^S)^2}{2\gamma(\bar{\sigma}_t^S)^2}dt+Z_t^{V}dB_t^V+\bar{Z}_td\bar{B}_t,\\[+0.5cm]
    Y_T&=&\displaystyle g(e^{\ln V_T},e^{\ln S_T}).
    \end{array}\right.
\end{eqnarray}
Note that FBSDE (\ref{chap5:MainFBSDE22}) is linear and the
coefficients satisfy Condition \ref{Cond2.4}, so by Lemma
\ref{LemmaEQU}, we know there exists a unique solution
$(Y,\mathbf{Z},\ln V,\ln S)\in\mathcal{C}([0,T];R)\times
H^2([0,T];R^2)\times\mathcal{C}([0,T];R)\times\mathcal{C}([0,T];R).$

Based on the solution $(Y,\mathbf{Z})$, we define a new probability
measure $\mathbf{P}$ by
$$\frac{d\mathbf{P}}{d\mathbf{Q}}=\mathscr{E}(N),$$
where $\mathscr{E}(N)$ is the Dol\'eans-Dade exponential of $N$ with
$$N=\int_0^{\cdot}\frac{\gamma}{2}Z_t^VdB_t^V+\int_0^{\cdot}\frac{\mu_t^S}{\bar{\sigma}_t^S}
d\bar{B}_t.$$ By Lemma \ref{Lemma2}, we know $\mathbf{P}$ is indeed
a probability measure. Under the new probability measure
$\mathbf{P}$, by Girsanov's theorem,
$\mathbf{W}=\mathbf{B}-[\mathbf{B},N]$ is a Brownian motion with
\begin{equation*}
\left\{
\begin{array}{ll}
W^V=B^V-\displaystyle\int_0^{\cdot}\frac{\gamma}{2}Z_t^Vdt,\\[+0.5cm]
\bar{W}=\bar{B}-\displaystyle\int_0^{\cdot}\frac{\mu_t^S}{\bar{\sigma}_t^S}dt.
\end{array}\right.
\end{equation*}

Under the probability measure $\mathbf{P}$ and with the Brownian
motion $\mathbf{W}$, let's rewrite the backward equation in FBSDE
(\ref{chap5:MainFBSDE22}):
\begin{align*}
dY_t=&-\displaystyle\frac{(\mu_t^S)^2}{2\gamma(\bar{\sigma}_t^S)^2}dt+Z_t^{V}\left(dW_t^V+\frac{\gamma}{2}Z_t^Vdt\right)
+
\bar{Z}_t\left(d\bar{W}_t+\frac{\mu_t^S}{\bar{\sigma}_t^S}dt\right)\\[+0.1cm]
    =&\ f_tdt+Z_t^VdW_t^V+\bar{Z}_td\bar{W}_t
\end{align*}
with $Y_T=g(V_T,S_T)$, and rewrite the forward equations in FBSDE
(\ref{chap5:MainFBSDE22}):
\begin{equation*}
dV_t/V_t=\mu_t^Vdt+\sigma_t^VdW_t^V+\bar{\sigma}_t^Vd\bar{W}_t,
\end{equation*}
with
\begin{equation*}
dS_t/S_t=\mu_t^Sdt+\bar{\sigma}_t^Sd\bar{W}_t.
\end{equation*}
Therefore the triple $(\Omega,\mathcal{F},\mathbf{P})$,
$\{\mathcal{F}_t\}$ and $(Y,\mathbf{Z},\mathbf{W})$ is just one weak
solution we want to find.\\

\noindent {\textbf{Acknowledgements.}} The research was supported in
part by EPSRC\ grant EP/F029578/1 and by the Oxford-Man Institute.

\bibliographystyle{plain}
\bibliography{myreference}

\noindent {\small \textsc{Gechun Liang, Terry Lyons and Zhongmin
Qian}}

\noindent{\small Oxford-Man Institute and Mathematical Institute}

\noindent{\small University of Oxford}

\noindent{\small Oxford OX2 6ED, U.K.}

\vskip0.3truecm

\noindent {\small {Email: \texttt{liangg@maths.ox.ac.uk;
tlyons@maths.ox.ac.uk; qianz@maths.ox.ac.uk}}}

\end{document}